\documentclass[leqno,12pt]{amsart}
\usepackage{amscd,amsfonts,amsmath,amssymb,amstext,amsthm}

\usepackage[utf8]{inputenc}
\usepackage[T2A]{fontenc}
\usepackage{cite}
\usepackage[unicode]{hyperref}
\usepackage{color}
\usepackage{xcolor}
\usepackage{comment}

\makeatletter
\def\@settitle{\begin{center}%
    \baselineskip14\p@\relax
    \bfseries
    \MakeUppercase{\@title}
  \end{center}%
}
\makeatother

\newtheorem{theorem}{Theorem}
\newtheorem{lemma}{Lemma}[section]
\newtheorem{proposition}{Proposition}
\newtheorem{corollary}{Corollary}[section]
\newtheorem{definition}{Definition}
\newtheorem{example}{\it Example}[section]

\numberwithin{equation}{section}

\def\vol{\:{\rm vol}}
\def\R{{\mathbb R}}

\begin{document}
\title{Average number of solutions}
\author{B. Kazarnovskii}
\address {Institute for Information Transmission Problems\newline
{\it kazbori@gmail.com}.}
\begin{abstract}
Let $X$ be an $n$-dimensional manifold and $V _1,\ldots,V_n\subset C^\infty(X,\mathbb R)$  finite-dimensional vector spaces.
For systems of equations
$\{f_i = a_i\colon\: f_i\in V_i,\:a_i \in\mathbb R,\:i=1,\ldots,n\}$
we discover a relationship between the average number of their solutions
and mixed volumes of convex bodies.
To do this, we choose Banach metrics in the spaces $V_i$.
Using these metrics,
we construct 1) the measure in the space of systems, and
2) Banach convex bodies in $X$, i.e., collections
of centrally symmetric convex bodies in the fibers of the cotangent bundle of $X$.
It turns out that the average number of solutions is equal
to the mixed symplectic volume of Banach convex bodies.
Earlier this result was obtained for Euclidean metrics in spaces
$V_i$.
In Euclidean case, the Banach convex bodies are the collections
of ellipsoids.
\end{abstract}
\renewcommand{\subjclassname}
{\textup{2010} Mathematics Subject Classification}
\maketitle
\section{Introduction}\label{intro}
Let $X$ be an $n$-dimensional  manifold,
 $V_1,\ldots, V_n$  be finite dimensional vector subspaces in $C^\infty (X,\mathbb R)$,
 and $V^*_i$ be their dual vector spaces.
We consider
the systems of equations
\begin{equation}\label{eqsys}
f_1 - a_1 =  \ldots= f_i - a_i= \ldots =f_n - a_n = 0,
\end{equation}
$f_i \in V_i,\, f_i \ne 0,\,a_i \in \R$.
Let $H_i=\{v^*\in V^*_i\:\vert\: v^*(f_i)=a_i\}$ be an affine hyperplane in $V^*_i$,
corresponding to equation $f_i-a_i=0$,
and $H=(H_1,\ldots,H_n)$ be a tuple of hyperplanes,
corresponding to system (\ref{eqsys}).
For an open set $U\subset X$,
denote by $N_U(H)$
the number of solutions of system (\ref{eqsys}) contained in $U$.
For a measure $\Xi$ on the set of tuples $(H_1,\ldots,H_n)$,
we define \emph{the average number of solutions} $\mathfrak M(U)$
as an integral of
$N_U(H)$ with respect to $\Xi$.

Let's now select some Banach metrics in the spaces $V_i$.
These metrics are used,
firstly,
to construct the measure $\Xi$
and, secondly,
to construct the
Banach convex bodies $\mathcal B_i$ in $X$.
Banach convex body or $B$-body
in $X$ is a collection of centrally symmetric
convex bodies in the fibers of the cotangent bundle of $X$.
The procedure for the appearance of $B$-bodies is described in Definition \ref{dfbBody}.

Using the standard symplectic structure on the cotangent
bundle,
we define the mixed symplectic volume of $B$-bodies,
see Definition \ref{dfMixed}.
Denote by $\mathcal B^U_i$ the restrictions of $B$-bodies $\mathcal B_i$
to the open set $U\subset X$.
It turns out that $\mathfrak M(U)$ equals to
mixed volume of $B$-bodies $\mathcal B^U_i$.
We refer this result as a smooth version of the BKK theorem \cite{B75}.

We choose the measure $\Xi$ on the space of systems of equations,
equal to a product of Crofton measures in spaces $V^*_i$.
Recall that a translation invariant measure on the Grassmanian of affine hyperplanes in
Banach space is called the Crofton measure,
if a measure of a set of hyperplanes,
crossing any segment,
equals to the length of this segment.
Under certain smoothness conditions
there exists a unique such measure.

If the Banach metric is Euclidean,
then the Crofton measure is invariant  under rotations.
In the Euclidean case
$B$-bodies $\mathcal B_i$
consist of ellipsoids,
see \cite{AK18,ZK14}.
From Nash embedding theorem \cite{N56} it follows
that
any collection of ellipsoids in $X$
can be obtained in this way.
From the Banach analogue of Nash theorem proved in \cite{BI94} it follows that
Banach convex bodies $\mathcal B_i$ arising in our situation are also arbitrary.

However,
if the unit ball in Banach space
is not a zonoid,
then the Crofton measure
\emph {is not positive everywhere},
see \cite{SCHN06}.
A brief explanation of the effect of non-positivity is given in Subsection \ref{mainBKK1}.
Recall that the zonotope is a polyhedron, represented as the Minkowski sum of segments,
and the zonoid is a limit of some converging sequence of zonotopes.
If the unit ball of the Banach metric is a zonoid,
then we call this metric a zonoid metric\footnote{In \cite{SCHN06} for the zonoid metric the term hypermetric is used}.
All ellipsoids are zonoids and,
respectively,
Euclidean metrics are zonoid metrics.

The non-positivity of the measure $\Xi$
reduces the validity of the notion
of average number of solutions.
For this reason,
in order to avoid the non-positivity of the Crofton meaure,
we choose
the most natural positive measure $\Xi$
related to Banach structures of the spaces $V_i$,
see subsection \ref{mainBKK2}.
It turns out
that this measure is a product of Crofton’s measures,
corresponding to some zonoid symmetrizations
of Banach metrics of the spaces $V_i$.
In this case,
$\mathfrak M_U$ equals to 
mixed
volume of $B$-bodies,
corresponding to the symmetrized metrics of the spaces $V_i$.
Such $B$-bodies are zonoid collections,
see Corollary \ref{corZonoid}.
\section{Main results}\label{main}
Further we use the following notations:

- $V$ is a finite dimensional vector space;

- ${\rm Gr}_m(V)$ is the Grassmanian, whose points are vector subspaces
of dimension $m$ in $V$;

 - ${\rm D}_m(V)\subset \bigwedge^m(V)$ is the cone of decomposable $m$-vectors;

- ${\rm AGr}^m(V)$ is the affine Grassmanian, whose points are affine
subspaces of codimension $m$ in $V$.
\subsection{Preliminaries.}\label{mainPrel}
We start from some notions,
mainly from \cite{AK18}.
\subsubsection{Finsler and Banach convex bodies.}\label{mainFBb}
Let $X$ be a smooth manifold, $\dim X=n$.
Suppose that for every $x\in X$ we are given
a convex body ${\mathcal E}(x) \subset T_x^*$
depending continuously on $x\in X$.
 We call the collection ${\mathcal E} =\{{\mathcal E}(x) \ \vert\ x\in X\}$ a Finsler set
in $X$.
The volume $\vol(\mathcal E)$  of a Finsler set $\mathcal E$
is defined as the volume
of $\cup_{x\in X}\mathcal E(x) \subset T^*X$ with respect
to the standard symplectic structure on the cotangent bundle.
More precisely,
if the symplectic form is $\omega$ then the volume form is $\omega ^n/ {n!}$.
Using Minkowski sum and homotheties,
we consider linear combinations of convex sets with non-negative coefficients.
The linear combination of Finsler sets is defined by
$$(\sum _i \lambda _i{\mathcal E}_i)(x) = \sum _i \lambda _i {\mathcal E}_i(x).$$
The volume $\vol(\lambda_1\mathcal E_1+\ldots+\lambda_n\mathcal E_n)$
is a homogeneous polynomial of degree $n$ in $\lambda_1, \ldots, \lambda _n$.
\begin{definition}\label{dfMixed}
The coefficient of polynomial $\vol(\lambda_1\mathcal E_1+\ldots+\lambda_n\mathcal E_n)$
at $\lambda_1\cdot\ldots\cdot\lambda_n$ divided by $n!$
is called the mixed volume of Finsler sets $\mathcal E_1,\ldots,\mathcal E_n$ and
is denoted by $\vol(\mathcal E_1, \ldots ,\mathcal E_n)$.
\end{definition}
\noindent
Given a $C^\infty $ map $f\colon X \to Y$ and
a Finsler set ${\mathcal E}$ in $Y$,
one can define the pull-back $f^*\mathcal E$ as follows.
\begin{definition}\label{dfbBodyPullBack}
$f^*{\mathcal E}(x) = d^*f(\mathcal E(f(x))$,
where $d^*f$
is the dual to the differential of $f$.
\end{definition}
\noindent
If the bodies $\mathcal E(x)$ are centrally symmetric,
then we call $\mathcal E$ Banach convex body
or $B$-body.
Pull-back of $B$-body also is $B$-body.
The procedure for appearance of $B$-bodies
in the problem is as follows.
\begin{definition}\label{dfbBody}
Let $V\subset C^\infty(X)$ be a finite dimensional Banach space,
$V^*$ is dual to $V$, and
$B\subset V$ is a ball of radius $1$ centered at $0$.

{\rm(1)}  The collection
  $\mathcal R=\{\mathcal R(v^*)=B\:\vert\: v^*\in V^*\}$
 is called the standard $B$-body in $V^*$.

{\rm(2)}
Define the mapping $\theta_V\colon X\to V^*$,
as $\theta_V(x)\colon f\mapsto f(x)$.

{\rm(3)}
We call the $B$-body
$\mathcal B=\theta_V^*(\mathcal R)$ a $B$-body
corresponding to the space of functions $V$.
\end{definition}
\subsubsection{Normal measures and normal densities.}\label{mainNm}
Recall that an $m$-density on $V$ is a continuous function
$\delta\colon{\rm D}_m (V) \to\mathbb R$,
such that $\delta (t\xi) = \vert t \vert \delta(\xi)$ for all $t\in\mathbb R$.
We can consider the $m$-density $\delta$ as a function on ${\rm Gr}_m (V)$,
whose value at the point $H\in{\rm Gr}_m (V)$ is a Lebesgue measure in $H$.
One example of an $m$-density is the absolute value of an exterior $m$-form.
An $m$-density on a manifold $X$ is an $m$-density $\delta _x$ on each tangent space $T_x$,
such that the assignment $x\mapsto \delta _x$ is continuous.
Let's now define the normal measure and
the normal density on a manifold.
\begin{definition}\label{dfNormaM}
A signed Borel measure $\mu$ on ${\rm AGr}_m(V)$ is called normal,
if $\mu$ is translation invariant and finite on compact sets.
 The space of normal measures on ${\rm AGr}_m(V)$ is denoted by $\mathfrak m_m(V)$.
\end{definition}
\noindent
We define a linear mapping $\chi_m$ of the space $\mathfrak m_m (V)$
into the space of $m$-densities as follows.
\begin{definition}\label{dfchi}
Let $\Pi_\xi\subset V$ be an
$m$-dimensional parallelotope generated by $\xi_1,\ldots,\xi_m\in V$ and
$
\mathcal J_{m,\Pi_\xi} = \{H \in {\rm AGr}^m(V) \ \vert \ H \cap \Pi_\xi \ne \emptyset\}.
$
For $\mu\in\mathfrak m_m$ put
$\chi _m(\mu)(\xi_1\wedge \ldots \wedge \xi_m ) = \mu ({\mathcal J}_{m,\Pi_\xi}).$
\end{definition}
\begin{definition}\label{dfNormalD}
The subspace $\chi_m(\mathfrak m_m(V))$ of the space of $m$-densities
is denoted by $\mathfrak n_m (V)$.
The densities from $\mathfrak n_m (V)$ are called normal $m$-densities
on $V$.
The normal density on the manifold $X$ is the density
whose values on $T_x (X)$ are normal for any $x\in X$.
\end{definition}
Further, we temporally create a scalar product in the space $V$
and corresponding Haar measure $d\nu$ on Grassmanian ${\rm Gr}_m(V)$.

For any function
$\varphi\in C^0({\rm Gr}_m(V))$
we construct
$\mu_\varphi\in\mathfrak m_m(V)$
as follows.
Let $A\subset{\rm AGr}^m(V)$
and $A_H =\{h\in H \colon\: h+H^\bot\in A\}$,
where $H^\bot$ is the orthogonal complement
of $H\in{\rm Gr}_m(V)$.
Define $\mu_\varphi(A)$ as
\begin{equation}\label{eqFunctMeasure}
  \mu_\varphi(A)=\int_{H\in{\rm Gr}_m(V)}
\varphi(H)
\left(\int_{h\in A_H}dh\right)\:d\nu(H).
\end{equation}
If $\varphi\in C^\infty({\rm Gr}_m(V))$,
then the normal measure $\mu_\varphi$ is called smooth.

For $\psi\in C^0({\rm Gr}_m(V))$
define the $m$-density 
$d_\psi(\xi_1\wedge\ldots\wedge\xi_m)=\psi(H_\xi)\vol_m(\Pi_\xi)$,
where $H_\xi$ is the subspace generated by the vectors $\xi_i$,
and $\vol_m(\Pi_\xi)$ is an $m$-dimensional area of the parallelotope $\Pi_\xi$.
Then $\psi\mapsto d_\psi$ \emph{is a  bijective mapping} from
$C^0({\rm Gr}_m(V))$ to the space of $m$-densities on $V$.
Smoothness of $m$-density $d_\psi$
is equivalent to the smoothness of $\psi$.
For $\phi\in C^0\left({\rm Gr}_m(V)\right)$ we set
\begin{equation}\label{eqt_m}
  t_m(\varphi)=\psi,
\end{equation}
where $d_\psi=\chi_m(\mu_\varphi)$.

Next, we use the cosine transform of functions on the Grassmannian.
Recall that the cosine transform
is
\begin{equation}\label{eqT_m}
 T_m(f) (G)=\int_{{\rm Gr}_m(V)}\vert\cos(H,G)\vert\: f(H)\:d\nu(H),
\end{equation}
where $\vert\cos(H,G)\vert$ is the distortion coefficient
of the $m$-dimensional area
under the mapping of orthogonal projection $H\to G$.
%
\begin{corollary}\label{corCosine} $t_m=T_m$.
\end{corollary}
\begin{proof}
Follows from
the above definitions of normal measure $\mu_\varphi$
and transforms $t_m$, $T_m$; see Definition \ref{dfchi} and equations
(\ref{eqFunctMeasure}), (\ref{eqt_m}), (\ref{eqT_m}).
\end{proof}
\begin{corollary}\label{cor1-1}
Let
$\delta$ be a smooth $1$-density on $V$.
Then
there exists a unique smooth normal measure $\mu\in\mathfrak m_1(V)$,
such that $\delta=\chi_1(\mu)$.
\end{corollary}
\begin{proof}
It is known (see, for example, \cite{AB04}),
that the restriction of $T_1$ to $C^\infty $ functions
is an automorphism $T_1: C^\infty ({\rm Gr}_1(V)) \to C^\infty({\rm Gr}_1(V)).$
Now Corollary \ref{cor1-1} follows
from Corollary \ref{corCosine}.
\end{proof}
\subsection{Smooth versions of theorem BKK}\label{mainBKK}
Let $V_1,\ldots,V_n$ be finite dimensional Banach spaces of smooth functions on $X\,$,
and let $V^*_i $ be their respective dual Banach spaces.
Suppose that the \emph{unit balls $B_i\subset V_i,\,B^*_i\subset V^*_i$ are
smooth convex bodies}.
We consider systems of equations of the form (\ref{eqsys}).
Recall that we identify such systems with points $H=(H_1,\ldots,H_n)$
of ${\rm AGr}^1(V^*_1)\times\ldots\times{\rm AGr}^1(V^*_n)$.
For any relatively compact open set $U\subset X$ let's
denote by $N_U(H)$ the number of isolated solutions of the corresponding
system in $U$.
For a measure $\Xi$ on ${\rm AGr}^1(V^*_1)\times\ldots\times{\rm AGr}^1(V^*_n)$,
denote by $\mathfrak M(U)$ the integral
of the function $N_U(H)$
with respect to the measure $\Xi$.
%
Let's call $\mathfrak M(U)$
\emph{the average number of solutions} of {\rm(\ref{eqsys}) in $U$}
with respect to a measure $\Xi$.
%
%
\subsubsection{The first version of BKK theorem}\label{mainBKK1}
Let's define $1$-density $\vol_{1,i}$ on
$V^*_i$,
as $\vol_{1,i}(\xi)=\|\xi\|^*_i$,
where $\|\,\|^*_i$ is the norm in the space $V^*_i$.
Then,
according to Corollary \ref{cor1-1},
there exists a unique smooth measure $\mu_i\in\mathfrak m_1(V_i^*)$,
such that $\vol_{1, i}=\chi_1(\mu_i)$.
%
\begin{theorem}\label{thmBKK_1}
Let $\mathcal B_i$ be a $B$-body in $X$,
corresponding to the space of functions $V_i$,
and let $\mathcal B_i^U$ be a restriction of $\mathcal B_i$ to $U$.
Then
\begin{equation}\label{eqBKK_1}
  \mathfrak M(U)=\frac{n!}{2^n}\vol(\mathcal B^U_1,\ldots,\mathcal B^U_n),
\end{equation}
where $\vol(\mathcal B^U_1,\ldots,\mathcal B^U_n)$ is a mixed volume of $B$-bodies.
\end{theorem}
\noindent
Recall that if
Banach balls $B_i$ are not zonoids,
then the measure $\Xi$ may be non-positive.
Any ellipsoid is a zonoid.
Hence, in the case of Euclidean metrics,
the measure $\Xi$ is positive.
Therefore, the theorem \ref{thmBKK_1}
coincides with Theorem 1 in \cite{AK18},
see also \cite{ZK14}.
Here is a brief explanation of non-positivity of the measure $\Xi$.
For more information see \cite{SCHN06}.
\begin{proposition}\label{prPositivity}
Let $V$ be a finite dimensional Banach space,
and $V^*$ be a dual vector space.
Suppose that the unit ball $B\subset V$ is smooth.
Denote by $\mu$ a smooth normal measure on ${\rm AGr}^1(V^*)$,
such that
$$
\forall \xi\in V^*\colon\,\chi_1(\mu)(\xi)=\|\xi\|^*.
$$
Then the measure $\mu$ is positive if and only if $B$ is a zonoid.
\end{proposition}
\begin{proof}
Let the space $V^*$ have a scalar product $\langle \cdot, \cdot \rangle$.
Then by applying (\ref{eqFunctMeasure})
and Corollaries \ref{corCosine} and \ref{cor1-1},
we obtain

(i) $\mu=\mu_\varphi$,
where $\varphi\in C^\infty\left({\rm Gr}_1(V^*)\right)$

(ii) $\|\xi\|^*=T_1(\varphi)(\xi)$.
\par\smallskip
\noindent
From (\ref{eqFunctMeasure}) it follows that
the positivity of measure $\mu$ is equivalent
to the positivity of function $\varphi$.

Let $S$ be a unit sphere centered at $0$ in the Euclidean space $V^*$.
Consider the standard mapping $S\to{\rm Gr}_1(V^*)$
and denote by $f$ a pull-back of function $\varphi$.
It follows from (ii), that
$$
\|\xi\|^*=\int_S\vert\langle\xi,y\rangle\vert\: f(y)\:ds(y),
$$
where $ds$ is the orthogonally invariant measure on $S$.
Now, since $\|\xi\|^*$
is a support function of the ball $B$,
then Proposition \ref{prPositivity} follows from the Proposition \ref{prSchn} below.
\end{proof}
\begin{proposition}\label{prSchn}{\rm (see \cite{SCHN13})}
Let
$f\in C^\infty(S)$, $f\geq0$,
\begin{equation}\label{eqzon}
 h(x)=\int_S\vert\langle x,y\rangle\vert f(y)\:ds(y).
\end{equation}
Then
$h(x)$
is a support function of some smooth zonoid
with the center of symmetry at the point $0$.
Conversely, the support function of any smooth zonoid centered at $0$
can be represented as in {\rm(\ref{eqzon})}.
\end{proposition}
\subsubsection{The second version of theorem BKK}\label{mainBKK2}
To state the theorem,
we use the notions of 1) the "natural measure"\ $\Xi(V_1,\ldots,V_n)$
on the manifold ${\rm AGr}^1(V^*_1)\times\ldots\times{\rm AGr}^1(V^*_n)$
and 2) "zonoid symmetrization"\ of the Banach norm.
These definitions will be given immediately after
the statement of theorem.
The difference between Theorem 2 and Theorem 1 is
that Theorem 2 uses symmetrized metrics of spaces $V_i$,
which leads to a nonnegativity of the measure $\Xi$.
The space $V$ with a symmetrized norm is denoted by $V_{\rm symm}$.
The unit ball of the space $V_{\rm symm}$ is a zonoid,
see Definition \ref{dfSymm} and Corollary \ref{corZonoid}.

Let $\mathfrak M(U)$ be an average number of solutions of systems
(\ref{eqsys}) contained in $U$ with respect to a measure $\Xi(V_1,\ldots,V_n)$
(see Subsection \ref{mainBKK}),
and ${\mathcal B}_{i,{\rm symm}}$ be a $B$-body,
corresponding to $V_{i,{\rm symm}}$.
\begin{theorem}\label{thmBKK_2}
Let $\mathcal B^U_{i,{\rm symm}}$ be a restriction of $\mathcal B_{i,{\rm symm}}$ to $U$.
Then
$$
\mathfrak M(U)=
\frac{n!}{2^n}\vol({\mathcal B}^U_{1,{\rm symm}},\ldots,{\mathcal B}^U_{n,{\rm symm}})
$$
\end{theorem}
%
%
Let's now define the natural measure $\Xi(V_1,\ldots,V_n)$.
If the measure $\Xi(V) $ is defined,
then we put
$
\Xi(V_1,\ldots,V_n) = \Xi(V_1) \times \ldots \times \Xi (V_n).
$
Hence, it remains to define the natural measure
$\Xi$ for $n=1$.

The manifold $Y$ is called a Banach manifold,
if in every tangent space $T_y(Y)$
there is a norm
continuously dependent on $y$.
Let $\mathfrak l(y)$ be a Lebesgue measure in $T_y(Y)$,
such that $\mathfrak l(y)(B_y) = 1$,
where $B_y\subset T_y(Y)$ is a ball of radius $1$ centered at $0$.
Denote by $d\mathfrak l$ the corresponding
volume density on $Y$.

Let $V$ be a Banach space with the unit sphere $S=\partial B$.
We restrict the norm function
to the tangent spaces of the sphere and
consider $S$
as a Banach manifold.
The following notation is used in the definition of a natural measure:

- $x\in S$

- $x^\bot\subset V^*$ is a
subspace of codimension $1$
orthogonal to
$x$

- ${\rm AGr}^1(V^*)\ni H(x,t)=\{y+x^\bot\:\vert\:t\in\R,\:\langle x,y\rangle=t\}$

- $A_x=\{H\in A\subset{\rm AGr}^1(V^*)\:\vert\:H=H(x,t),\:t\in\R\}$,

- $d\mathfrak l$ is a Banach volume density in $S$.
\begin{definition}\label{dfNatural}
We define the natural measures $\Xi(V)$ in ${\rm AGr}^1(V^*)$ and
$\Xi(V_1,\ldots,V_n)$ in
${\rm AGr}^1(V_1^*)\times\ldots\times{\rm AGr}_1(V_n^*)$
as
\begin{equation}\label{eqNatural}
  \Xi(V)(A)=\frac{1}{2}\int_{x\in S}\left(\int_{H(x,t)\in A_x}\:dt\right)\:d\mathfrak l(x)
\end{equation}
and $\Xi(V_1,\ldots,V_n)=\Xi(V_1) \times \ldots \times \Xi (V_n)$.
\end{definition}
Let's now define the zonoid symmetrization $V_{\rm symm}$
of Banach space $V$.
Let
\begin{equation}\label{eqSymm}
  h_{\rm symm}(y)=\frac{1}{2}\int_{x\in S} \vert\langle x,y\rangle\vert\:d\mathfrak l(x).
\end{equation}
The function $h_{\rm symm}$ is
smooth, positive, even, convex and positively homogeneous of degree $1$,
i.e. $h_{\rm symm}(tx)=\vert t\vert\: h_{\rm symm}(x)$.
Hence
$h_{\rm symm}$ is a support function of a smooth centrally symmetric
convex body $B_{\rm symm}\subset V$.
We consider $B_{\rm symm}$ and $h_{\rm symm}$ as
symmetrizations of Banach unit ball $B\subset V$ and
its support function respectively.
\begin{definition}\label{dfSymm}
Let $\|\:\|_{\rm symm}$ be a norm in space $V$ with a unit ball $B_{\rm symm}$.
The corresponding Banach space is denoted by $V_{\rm symm}$.
We call the space $V_{\rm symm}$
a zonoid symmetrization of $V$.
\end{definition}
\begin{corollary}\label{corSymm}
The natural measure $\Xi(V)$ is a Crofton measure in Banach space $V^*_{\rm symm}$
dual to $V_{\rm symm}$.
\end{corollary}
\begin{proof}
Let $A$ be a set of affine hyperplanes in $V^*$
intersecting the segment $[0,\xi]$.
From (\ref{eqNatural}) and (\ref{eqSymm}) it follows that
$\Xi(V)(A)=h_{\rm symm}(\xi)$.
Now the assertion follows from the equality Let $\|\:\|^*_{\rm symm}=h_{\rm symm}$,
 where $\|\:\|^*_{\rm symm}$ is a norm in Banach space $V^*_{\rm symm}$
dual to $V_{\rm symm}$.
\end{proof}
\noindent
\begin{corollary}\label{corZonoid}
$B_{\rm symm}$ is a zonoid.
\end{corollary}
\begin{proof}
Follows from the Proposition \ref{prSchn}.
\end{proof}
%
%
\begin{corollary}\label{corZonoid_i}
Let ${\mathcal B}_{i,{\rm symm}}$ be a $B$-body corresponding to $V_{i,{\rm symm}}$.
Then for any $x\in X$ the convex body ${\mathcal B}_{i,{\rm symm}}(x)$
is a zonoid.
\end{corollary}
\begin{proof}
From Definitions \ref{dfbBody},(3) and \ref{dfbBodyPullBack} it follows
that $\mathcal B_{i,{\rm symm}}(x)$ is the image of zonoid $B_{i,{\rm symm}}$
under a linear mapping $V_{i,{\rm symm}}\to T^*_x(X)$.
Hence,
$\mathcal B_{i,{\rm symm}}(x)$ is a zonoid.
\end{proof}
%
Note that the symmetrization of  Euclidean metric differs from the original one
by a constant factor.
\section{Proofs of main resuts}\label{proof}
\subsection{Banach convex bodies and normal densities}\label{proof0}
Here are some facts about
normal densities and Banach convex bodies, see \cite{AK18}.
\subsubsection{The product of normal densities}\label{proof01}
The graded vector space
$$\mathfrak m(V)=\mathfrak m_0(V)\oplus\mathfrak m_1(V)\oplus\ldots \oplus\mathfrak m_n(V)$$
(see  Definition \ref{dfNormaM})
has a structure of a graded ring.
Namely,
let
$$
{\mathcal D}_{p,q} = \{(G,H)\in{\rm AGr}^p(V)\times{\rm AGr}^q(V)\ \vert \ {\rm codim}\, G\cap H \ne p+q\}.
$$
Then we have the mapping
$$
P_{p,q}\colon{\rm AGr}^p(V)\times{\rm AGr}^q(V)\setminus\mathcal D_{p,q}\to
{\rm AGr}^{p+q}(V),
$$
given by $P_{p,q}(G,H) = G\cap H$.
\begin{definition}\label{dfProdMes}
Let $\mu\in\mathfrak m_p(V), \, \nu\in\mathfrak m_q(V), \, p+q \le n$.
For $D\subset{\rm AGr}^{p+q}$ we put
$$\mu\cdot\nu(D) = (\mu \times \nu)(P_{p,q}^{-1}(D))$$
\end{definition}
\noindent
Let $\chi = \oplus \chi_m\colon\mathfrak m(V)\to\mathfrak n(V)$;
see Definition \ref{dfchi}.
It is easily seen that ${\rm Ker}\:\chi$
is a homogeneous ideal of the graded ring $\mathfrak m(V)$.
Therefore $\mathfrak n(V)$ carries a structure of a graded ring,
the ring of normal densities on $V$.
The pointwise construction of product leads
to the definition of \emph{the ring of normal densities on smooth manifold} $X$,
denoted by ${\mathfrak n}(X)$.
\begin{example}
Let $w$ be a differential form of degree $m$ on a manifold $X$.
Then the absolute value  $\vert w\vert$ is an $m$-density on $X$.
If the form $w$ is locally decomposable,
i.e.
if $w$ is the product of 1-forms in a neighborhood of every point $x\in X$,
then $\vert w\vert\in\mathfrak n(X)$.
If $w_1, w_2$ are locally decom\-posable forms then
$\vert w_1 \wedge w_2\vert = \vert w_1 \vert \cdot \vert w_2\vert$;
see \cite{AK19}.
\end{example}
\subsubsection{Contravariance}\label{proof02}
The assignment $X \to {\mathfrak n}(X)$ is
a contravariant functor from the category of smooth manifolds to the
category of commutative graded rings, see \cite{AK18}, Theorem 6.
Namely, the pull-back of a normal density is normal
and the product of pull-backs is the pull-back of the product of normal densities.
This property is based on the following
construction of {\it the pull-back of a normal measure}.

For a linear map $\varphi\colon U\to V$ of vector spaces
one can define a ring homomorphism
$\varphi^*\colon\mathfrak m(V)\to\mathfrak m(U)$,
so that $\chi _m (\varphi^*\mu ) = \varphi^*(\chi _m (\mu ))$.
We will further only consider $\varphi^*$ for $\varphi$ being an epimorphism.
In this case we have a closed embedding
$\varphi_*:{\rm AGr}^m(V) \to {\rm AGr}^m(U)$
defined by taking the preimage under $\varphi $ of an affine subspace in $V$. By definition, the measure $\varphi ^*\mu$ is
supported on $\varphi _*{\rm AGr}^m(V)$ and
$
\varphi^*\mu(\Omega) = \mu (\varphi _*^{-1}(\Omega ))$ for any Borel set $\Omega \subset \varphi_*{\rm AGr}^m(V)$.
\subsubsection{Normal densities and convex bodies}\label{proof03}
Let $A_1\ldots, A_m$ be convex bodies in the dual space $V^*$. The $m$-density
$d_m(A_1, \ldots, A_m)$ on $V$ is defined as follows.
\begin{definition}\label{dfd_i}
Let $H$ be the subspace of $V$ generated by $\xi _1, \ldots, \xi_m\in V$,
$H^\bot \subset V^*$  the orthogonal complement to $H$,
and $\pi _H\colon V^* \to V^*/H^\bot$ the projection map.
Consider $\xi_1\wedge \ldots \wedge \xi_m$ as a volume form on $V^*/H^\bot $.
Then $d_m(A_1,\ldots,A_m)(\xi_1,\ldots,\xi_m)$ is the mixed
$m$-dimensonal volume of $\pi_HA_1, \ldots, \pi_HA_m$.
\end{definition}
\noindent
In particular,
$d_1(A)(\xi) = h(\xi) - h(-\xi )$, where $h : V \to { R}$ is the support
function of a convex body $A$.
\begin{lemma}\label{lm3-1}
Let $A\subset V^*$ be a smooth convex body, $\dim A=m\leq\dim V$.
Then $d_1(A)\in\mathfrak n_1(V)$.
\end{lemma}
\begin{proof}
For $m=\dim V$ Lemma follows from Corollary \ref{cor1-1}.
Let $A$ be contained in $m$-dimensional subspace $H$ of $V^*$.
(If $m<\dim V$ then the density $d_1(A) $ is not smooth
and the corollary \ref{cor1-1} is inapplicable.)
Then the support function of $A$
is a pull-back of a smooth function on ${\rm D}_1(V/H^\bot)$
under the mapping of projection $V\to V/H^\bot$,
where $H^\bot\subset V$ is the orthogonal complement of the subspace $H$.
Hence, $d_1(A)$ is a pull-back of the smooth $1$-density
and therefore is normal.
\end{proof}
\noindent
The proof of the theorems \ref{thmBKK_1}, \ref{thmBKK_2} is
based on the following identity;
see \cite{AK18}, Theorem 7.
Suppose $A_1, \ldots, A_m$ are centrally symmetric smooth convex bodies.
Then
\begin{equation}\label{eqEquality_1}
 d_1(A_1)\cdot \ldots \cdot d_1(A_m) = m!\,d_m(A_1, \ldots, A_m).
\end{equation}
From (\ref{eqEquality_1}) it follows
that $\forall m\colon\:d_m(A_1, \ldots, A_m)\in\mathfrak n_m(V)$.

For
$B$-bodies $\mathcal E_1,\ldots,\mathcal E_m$ denote by
$D_m(\mathcal E_1,\ldots,\mathcal E_m)\in\mathfrak n_m(X)$
the $m$-density whose value at $x\in X$
equals $d_m(\mathcal E_1(x),\ldots,\mathcal E_m(x))$.
By definition, $D_m(\mathcal E)=D_m(\mathcal E_1,\ldots,\mathcal E_m)$,
where $\mathcal E_i=\mathcal E,\: i=1,\ldots,m$.
For $n=\dim X$, from Definitions \ref{dfMixed}, \ref{dfd_i} it follows
that
\begin{equation}\label{eqMixedD}
  \int _X D_n(\mathcal E_1,\ldots,\mathcal E_n)=\vol(\mathcal E_1,\ldots,\mathcal E_n).
\end{equation}
The mixed volume of $B$-bodies $\vol(\mathcal E_1,\ldots,\mathcal E_n)$
(see Definition \ref{dfMixed}) relates to
the commonly used mixed volume of convex bodies in the following way.
Let $h$ be a Riemannian metric on $X$ and $h^*$
the dual metric on $T_x^*$.
Then
\begin{equation}\label{eqUsual}
 \vol(\mathcal E_1,\ldots,\mathcal E_n)=\int_XV(\mathcal E_1(x),\ldots,\mathcal E_n(x))\:dx
\end{equation}
where $V$ is a mixed volume with respect to the metric $h^*$
and
$dx$ is the Riemannian $n$-density on $X$.
\subsection{Crofton formula for product of Banach spaces}\label{proofCrofton}
Let $V_1,\ldots, V_n$  be finite dimensional Banach spaces,
 $V^*_1,\ldots,V^*_n$ their dual vector spaces,
let $\Pi_\xi$ be an $n$-dimensional parallelotope generated by $\xi_1,\ldots,\xi_n\in V^*_1\times\ldots\times V^*_n$,
and let
$$
\mathfrak P_\xi=\{H=(H_1\times\ldots\times H_n)\in
\prod_{i\leq n}{\rm AGr}^1(V_i^*)\:\colon\: \Pi_\xi\cap H\neq\emptyset\}.
$$
Let's define the Crofton $n$-density $\Omega$ on $V^*_1\times\ldots\times V^*_n$
as
$$
\Omega(\xi_1\wedge\ldots\wedge\xi_n)=(\mu_1\times\ldots\times\mu_n)(\mathfrak P_\xi),
$$
where $\mu_i$ is a Crofton measure in ${\rm AGr}^1(V^*_i)$;
see subsection \ref{mainBKK1}.
Below we keep the notation $\vol_{1, i}$, $\mu_i$ for the pull-backs of $\vol_{1, i}$, $\mu_i$ under the mappings of projection
$\pi_i\colon V^*_1\times\ldots\times V^*_n\to V^*_i$.
Let
$
\mu_1\cdot\ldots\cdot\mu_n$ and $\vol_{1,1}\cdot\ldots\cdot\vol_{1,n}$
be products of normal measures $\mu_i$
and normal densities $\vol_{1,i}$ in the rings
$
\mathfrak m(V^*_1\times\ldots\times V^*_n)$ and $\mathfrak n(V^*_1\times\ldots\times V^*_n)$
respectively.
\begin{proposition}\label{prCrofton1}
$\Omega=\vol_{1,1}\cdot\ldots\cdot\vol_{1,n}$
\end{proposition}
\begin{proof}
From the description of pull-back operation for normal measure
in Subsection \ref{proof02}
and from Definition \ref{dfProdMes}
it follows that
$$
\mu_1\cdot\ldots\cdot\mu_n=\mu_1\times\ldots\times\mu_n.
$$
%
Hence,
$$\Omega(\xi_1\wedge\ldots\wedge\xi_n)=(\mu_1\cdot\ldots\cdot\mu_n)(\mathfrak P_\xi)
=(\vol_{1,1}\cdot\ldots\cdot\vol_{1,n})(\xi_1\wedge\ldots\wedge\xi_n),$$
where the second equality follows from definition of product of normal densities,
see in Subsection \ref{proof02}.
\end{proof}
%
Let $X$ be a submanifold in $V^*_1\times\ldots\times V^*_n$, $\dim X=n$.
For $H=H_1\times\ldots\times H_n\in\prod_{i\leq n}{\rm AGr}^1(V_i^*)$
denote by $N_X(H)$
the number of points in $X\cap H$.
\begin{theorem}\label{thmCrofton1}
$
\int\limits_{H\in\prod_{i\leq n}{\rm AGr}^1(V_i^*)}N_X(H)\:d\Xi=
\int\limits_X \vol_{1,1}\cdot\ldots\cdot\vol_{1,n}.
$
\end{theorem}
\begin{proof}
Follows from Proposition \ref{prCrofton1}.
\end{proof}
\noindent
Let $B_j\subset V_j$ be a unit ball centered at $0$.
Denote by $\mathcal B_j$ the pull-back of the standard $B$-body $V^*_j\times B_j$
in $V^*_j$ under the mapping
$X\hookrightarrow\prod_{i\leq n}V^*_i\xrightarrow{\pi_j} V^*_j$,
see Definitions \ref{dfbBodyPullBack} and \ref{dfbBody},(3).
\begin{theorem}\label{thmCrofton2}
$
\int\limits_{H\in\prod_{i\leq n}{\rm AGr}^1(V_i^*)}N_X(H)\:d\Xi=
\frac{n!}{2^n}\vol(\mathcal B_1,\ldots,\mathcal B_n),
$
where $\Xi=\mu_1\times\ldots\times\mu_n$.
\end{theorem}
\begin{proof}
The $1$-density $\vol_{1,i}$ on the space $T_x(X)$
is a support function convex body $\mathcal B_j(x)\subset T^*_x(X)$.
Hence,
the restriction of $\vol_{1,j}$ on $X$
equals to $\frac{1}{2}D_1(\mathcal B_j)$,
see Definition \ref{dfd_i}.
Applying (\ref{eqEquality_1}) we have
 $$
 \vol_{1,1}\cdot\ldots\cdot\vol_{1,n}=\frac{1}{2^n}D_1(\mathcal B_1)\cdot\ldots\cdot D_1(\mathcal B_n)=\frac{n!}{2^n}D_n(\mathcal B_1,\ldots,\mathcal B_n).
 $$
 Now the Theorem follows from (\ref{eqMixedD}).
\end{proof}
\subsection{Proof of Theorem \ref{thmBKK_1}}\label{proof1}
Assume first that
the mapping
$$
\Theta=\theta_{V_1}\times\ldots\times\theta_{V_n}\colon\:U\to V^*_1\times\ldots\times V^*_n
$$
is a proper smooth embedding, see Defifnition \ref{dfbBody}, (2).
Then
Theorem \ref{thmBKK_1} follows from Theorem \ref{thmCrofton2}.

Assume now that the differential of the map $\Theta$
in $U$ is non-degenerate.
Then there exists a locally finite open cover $U=\bigcup_j U_j$,
such that the mapping $\Theta$ defines an embedding for each of $U_j$.
Now it remains to notice that both sides of (\ref{eqBKK_1})
are additive functions of the domain $U$.
Therefore, the theorem is proved by the inclusion-exclusion formula.
Recall
that the function $A$ is called additive,
if $A(U_1\cup U_2)=A(U_1) + A (U_2) - A (U_1\cap U_2)$.

Let $X_s\subset X$ be a set of points,
in which the differential of the map $\Theta$ is degenerate.
The Theorem \ref{thmBKK_1} follows from the Lemma.
\begin{lemma}
Both sides of {\rm(\ref{eqBKK_1})}
for domains $U$ and $U\cap(X\setminus X_s)$ coincide.
\end{lemma}
\begin{proof}
If $x\in X_s$
then there exists a proper subspace $E\subset T^*_xX$,
such that $\forall i\colon\mathcal B_i(x)\subset E$.
Hence, $\dim\left(\cup_{x\in X_s}\mathcal B_i(x)\right)<\dim T^*X$.
Therefore the left hand sides of (\ref{eqBKK_1})
for $U$ and $U\cap(X\setminus X_s)$ coincide.

Sard's lemma states that
the $n$-dimensional Hausdorff measure of $\Theta(X_s)$ is zero.
This implies the equality of the right-hand sides of (\ref{eqBKK_1})
(we omit the details).
\end{proof}
\subsection{Proof of Theorem \ref{thmBKK_2}}\label{proof2}
Let $\Xi(V_i)$ be the natural measure on ${\rm AGr}^1(V^*_i)$,
see Definition \ref{dfNatural}.
Let $V^*_{i,{\rm symm}}$ be a Banach space dual to the symmetrization
 $V_{i,{\rm symm}}$ of Banach space $V_i$.
The measure $\Xi(V_i)$ is a Crofton measure in Banach space $V^*_{i,{\rm symm}}$,
see Corollary \ref{corSymm}.
Hence, the Theorem \ref{thmBKK_2} coincides with the Theorem \ref{thmBKK_1}
for zonoid symmetrizations $V_{1,{\rm symm}},\ldots,V_{n,{\rm symm}}$
of spaces $V_{1},\ldots,V_{n}$.

\end {document}